\documentclass[11pt]{amsart}
\usepackage{amsfonts,amsthm,amsmath}
\usepackage{amscd}
\usepackage{natbib}
\usepackage{hyperref}

\newtheorem{theorem}{Theorem}[section]
\newtheorem{lemma}[theorem]{Lemma}
\newtheorem{prop}[theorem]{Proposition}

\newtheorem{remark}[theorem]{Remark}

\newtheorem{example}[theorem]{Example}

\newcommand{\prob}{\stackrel{P}{\longrightarrow}}
\newcommand{\eid}{\stackrel{d}{=}}
\newcommand{\one}{{\bf 1}}
\newcommand{\reals}{{\mathbb R}}
\newcommand{\bbr}{\reals}
\newcommand{\vep}{\varepsilon}
\newcommand{\bbz}{\protect{\mathbb Z}}

\newcommand{\qbin}{{\protect{\mathbb Q}}_{bin}}

\numberwithin{equation}{section}

\begin{document}

\title[Stationary S$\alpha$S random fields]{Group theoretic dimension of stationary symmetric $\alpha$-stable random fields}
\author[A. Chakrabarty]{Arijit Chakrabarty}
\address{Statistics and Mathematics Unit,
Indian Statistical Institute, Delhi, India}
\email{arijit@isid.ac.in}
\author[P. Roy]{Parthanil Roy}
\address{Statistics and Mathematics Unit,
Indian Statistical Institute, Kolkata, India}
\email{parthanil.roy@gmail.com}

\begin{abstract}
The growth rate of the partial maximum of a stationary stable process was first studied in the works of \cite{samorodnitsky:2004a, samorodnitsky:2004b}, where it was established, based on the seminal works of \cite{rosinski:1995, rosinski:2000}, that the growth rate is connected to the ergodic theoretic properties of the flow that generates the process. The results were generalized to the case of stable random fields indexed by $\bbz^d$ in \cite{roy:samorodnitsky:2008}, where properties of the group of nonsingular transformations generating the stable process were studied as an attempt to understand the growth rate of the partial maximum process.
This work generalizes this connection between stable random fields and group theory to the continuous parameter case, that is, to the fields indexed by $\bbr^d$.
\end{abstract}

\keywords{ Random field, stable process, extreme value theory, maxima, ergodic theory, group action.\vspace{.5ex}}
\thanks{Arijit Chakrabarty's research was partially supported by the Centenary Postdoctoral Fellowship at the Indian Institute of Science and Parthanil Roy's research was partially supported by a start up grant from the Michigan State University.}

\maketitle

\section{Introduction}\label{sec:intro}
This paper investigates the growth rate of the partial maxima of stationary symmetric stable non-Gaussian random fields indexed by $\bbr^d$. Let ${\bf X}:=\{X_t:t\in\bbr^d\}$ be a measurable and stationary symmetric $\alpha$-stable (S$\alpha$S) random field with $0<\alpha<2$. This means that for all $c_1, c_2,
\ldots, c_k \in \mathbb{R}$, and, $t, t_1,t_2,\ldots,t_k \in
\mathbb{R}^d$, $\sum_{j=1}^k c_j X_{t_j+t}$ follows a symmetric
$\alpha$-stable distribution that does not depend on $t$. For further reference on S$\alpha$S distributions and
processes, the reader is referred to \cite{samorodnitsky:taqqu:1994}.

Stationarity ensures that the law of the random field ${\bf X}$ is invariant
under the shift action of the group $\mathbb{R}^d$ on the index-parameter of the field. This group action, when viewed in the space of integral representations with respect to S$\alpha$S random measures (see \cite{samorodnitsky:taqqu:1994}), is not necessarily invariant but remains nonsingular. This was established in the seminal works of \cite{rosinski:1995} (for $d=1$) and \cite{rosinski:2000} (for $d>1$).

The nonsingular group action obtained in \cite{rosinski:1995, rosinski:2000} plays a very important role in the behavior of extremes of stationary S$\alpha$S random fields. This connection was first explored in the one-dimensional case in \cite{samorodnitsky:2004a, samorodnitsky:2004b}, which was later generalized for any $d \geq 2$ in \cite{roy:samorodnitsky:2008} and \cite{roy:2010b}. These works dealt with the partial maxima process of stationary S$\alpha$S random fields when the index parameter runs over a $d$-dimensional hypercube of length increasing to infinity. The rate of growth of this maxima process was exactly calculated when the underlying group action is not conservative. In the case of conservative actions, however, only an upper estimate on the rate of growth could be given in general.

In some discrete multiparameter cases, using the group theoretic structures of the underlying group action, a better estimate on the rate (sometimes the exact rate) of growth of the partial maxima sequence has been given in Section~5 of \cite{roy:samorodnitsky:2008}. The current paper extends this connection with group theory to the continuous parameter case by approximating the random field ${\bf X}$ by its discrete parameter skeletons.

This paper is organized as follows. In Section~\ref{sec:prelim}, we present some preliminaries. Section~\ref{sec:main:result} contains the main result of this paper and a couple of examples. The main result is then proved in Section~\ref{sec:proof} based on a bunch of ergodic theoretic Lemmas, whose proofs are presented in the Appendix.

\section{Preliminaries}\label{sec:prelim}

As mentioned above, nonsingular group actions play a significant role in the study of stationary stable random fields and hence, we start with a brief introduction to such actions. Let $(G,+)$ be a topological group with identity element $0$ and Borel $\sigma$-field $\mathcal{G}$, and $(S,\mathcal{S}, \mu)$ be a $\sigma$-finite measure space. A collection of measurable maps $\{\phi_t\}_{t \in G}$ on S is called a nonsingular $G$-action on $S$ if there exists $S^\prime \in \mathcal{S}$ with $\mu(S \setminus S^\prime)=0$ such that
\begin{enumerate}
\item $(t,s) \mapsto \phi_t(s)$ is a measurable map from $(G \times S^\prime,\, \mathcal{G} \otimes \mathcal{S}^\prime)$ to $(S^\prime, \mathcal{S}^\prime)$ (here $\mathcal{S}^\prime$ is the restriction of the $\sigma$-field $\mathcal{S}$ on $S^\prime$),
\item $\phi_0(s)=s$ for all $s \in S^\prime$,
\item $\phi_{t_1+t_2}(s)=\phi_{t_1} \circ \phi_{t_2}(s)$ for all $t_1, t_2 \in G$ and $s \in S^\prime$, and
\item $\mu \circ \phi_t \sim \mu$ for all $t \in G$ (here ``$\sim$'' denotes equivalence of measures).
\end{enumerate}
\noindent See, for instance, \cite{aaronson:1997}, \cite{krengel:1985}, \cite{varadarajan:1970} and \cite{zimmer:1984} for detailed discussions on nonsingular (also known as quasi-invariant) group actions.

If $G$ is countable then $W \in \mathcal{S}$ is called a wandering set if $\{\phi_t(W):\,t \in G\}$ is a pairwise disjoint collection and $\{\phi_t\}_{t \in G}$ is called conservative if it does not admit any wandering set of positive $\mu$-measure. On the other hand, for $G=\mathbb{R}^d$, it can be shown, based on a result of \cite{kolodynski:rosinski:2003}, that if the restriction $\{\phi_t\}_{t \in \mathbb{Z}^d}$ is conservative then so are $\{\phi_t\}_{t \in 2^{-i}\mathbb{Z}^d}$ for all $i=1,2,\ldots$ ; see Proposition~2.1 in \cite{roy:2010b}. Therefore, the group action $\{\phi_t\}_{t \in \mathbb{R}^d}$ can be defined to be conservative in this case.

We now present the connection between structures of stationary S$\alpha$S random fields and nonsingular group actions. It is known that any measurable S$\alpha$S random field $\mathbf{X}=\{X_t:t\in\bbr^d\}$ (not necessarily stationary) has an integral representation given by
\begin{equation}\label{intro.eq0}
\{X_t:t\in\bbr^d\}\eid\left\{\int_{S}f_t(s)\tilde M(ds):t\in\bbr^d\right\}\,,
\end{equation}
where $\tilde M$ is a S$\alpha$S random measure on some standard Borel space $(S,{\mathcal S})$ with a $\sigma$-finite control measure $\mu$, $f_t\in L^\alpha(S,\mu)$ for all $t\in\bbr^d$ and $(t,s) \mapsto f_t(s)$ is a jointly measurable map; see \cite{samorodnitsky:taqqu:1994} and \cite{rosinski:woyczynski:1986}. Without loss of generality we can assume that the full support condition
\begin{equation*}
support\left\{f_t:\,t \in \mathbb{R}^d\right\}=S \label{condition:full:support}
\end{equation*}
holds for any integral representation $\{f_t\}$ of $\mathbf{X}$.

The structure of stationary S$\alpha$S random fields has been studied in \cite{rosinski:1995, rosinski:2000}. In those works, it has been shown that the functions $f_t$ in \eqref{intro.eq0} can be chosen to be of the form
\begin{equation}\label{intro.eq1}
f_t(s)=c_t(s)\left(\frac{d\mu\circ\phi_t}{d\mu}(s)\right)^{1/\alpha}f\circ\phi_t(s),\,t\in\bbr^d\,,
\end{equation}
where $f\in L^\alpha(S,\mu)$, $\{\phi_t:t\in\bbr^d\}$ is a nonsingular group action of the group $\bbr^d$ on $S$ and $\{c_t:t\in\bbr^d\}$ is a cocycle, {\it i.e.}, $(t,s)\mapsto c_t(s)$ is a jointly measurable function from $\bbr^d\times S$ to $\{-1,1\}$ such that for all $u,v\in\bbr^d$, $c_{u+v}(s)=c_v(s)c_u(\phi_v(s))$ for $\mu$-a.a. $s\in S$. Conversely, $\{X_t: t\in\bbr^d\}$ defined by \eqref{intro.eq0} and \eqref{intro.eq1} is a stationary $S\alpha S$ random field.

\begin{remark} \label{remark:uniqueness:of:int:repn}
{\rm In fact, \cite{rosinski:1995, rosinski:2000} established that every minimal representation (see \cite{hardin:1982b}) of $\mathbf{X}$ is of the form \eqref{intro.eq1}. Although this Rosi\'{n}ski representation may not be unique, it has been established based on a rigidity result of $L^\alpha$ spaces (due to \cite{hardin:1981}) that if in one Rosi\`nski representation of $\mathbf{X}$, the underlying group action is conservative then so is the action in all Rosi\'{n}ski representations; see \cite{rosinski:1995, rosinski:2000}, \cite{roy:samorodnitsky:2008} and \cite{roy:2010b}. In other words, the spaces of stationary measurable S$\alpha$S random fields generated by conservative and nonconservative actions are disjoint.}
\end{remark}

Now we turn our attention to the extremes of stationary measurable S$\alpha$S random fields. To this end, we assume further that $\mathbf{X}$ is locally bounded (see, for example, \cite{samorodnitsky:taqqu:1994} for sufficient conditions for local boundedness of  $\mathbf{X}$). Since $\mathbf{X}$ is stationary and
measurable, it is continuous in probability by Proposition
3.1 in \cite{roy:2010b}. Therefore, as in the one-dimensional case in \cite{samorodnitsky:2004b}, we can take its separable version and define (avoiding the usual measurability problems) the finite-valued maxima process
\begin{equation}
M_t:=\sup_{s\in[-t\one,t\one]\cap\Gamma}|X_s|,\,t\ge 0, \label{def:M_t}
\end{equation}
where $\Gamma:=\cup_{n=0}^\infty\Gamma_n$ with $\Gamma_n:=2^{-n}\bbz^d$, $n\ge0$ and $[u,v]:=\{s \in \mathbb{R}^d:\,u \leq
s \leq v\}$ (the inequality should be interpreted componentwise).

As mentioned earlier, the rate of growth of $M_t$ was studied in \cite{samorodnitsky:2004b} and \cite{roy:2010b}, where it was established that as $t \to \infty$,
\[
t^{-d/\alpha} M_t \Rightarrow \left\{ \begin{array}{ll}
                0   & \mbox{ if $\{\phi_t\}_{t \in \mathbb{R}^d}$ is conservative},\\
                \mbox{Fr\'echet distribution} & \mbox{ if $\{\phi_t\}_{t \in \mathbb{R}^d}$ is not conservative,}
             \end{array}
     \right.
\]
where $\{\phi_t\}_{t \in \mathbb{R}^d}$ is as in \eqref{intro.eq1}. Note that the above facts are in agreement with Remark~\ref{remark:uniqueness:of:int:repn}. This phase transition can be argued to be a transition from longer to shorter memory as described in \cite{samorodnitsky:2004a}. In particular, this means that only an upper bound can be given on the rate of growth of the partial maxima process $M_t$ when the underlying action is conservative.

For the discrete parameter case when the underlying $\mathbb{Z}^d$-action is conservative, depending on the group theoretic properties of the underlying action, a better estimate of the rate of growth of the partial maxima sequence
$$
M^\prime_n:=\max_{s\in[-n\one,n\one]\cap\mathbb{Z}^d}|X_s|,\; n \ge 0\
$$
was given in \cite{roy:samorodnitsky:2008}. This work had the following key idea: instead of looking at $\{\phi_t\}$ as a $\mathbb{Z}^d$-action, look at it as a group
$$
A=\{\phi_v:\,v \in \mathbb{Z}^d\}
$$
of nonsingular transformations on $S$ in order to remove the redundancy in the action. This group $A$ happens to be a quotient group of $\mathbb{Z}^d$ and hence, by the structure theorem of finitely generated Abelian groups (see, for example, \cite{lang:2002}), can be decomposed as a direct sum of two subgroups, one of which is a free Abelian group $\bar{F}$ and the other is a finite group $\bar{N}$. The subgroup $\bar{N}$ corresponds to the cycles in the action and $\bar{F}$ being a free Abelian group has an isomorphic copy $F$ sitting inside $\mathbb{Z}^d$ which can be thought of as the effective index set of the random field $\{X_t\}$. See Section~\ref{sec:main:result} below for the details.

\cite{roy:samorodnitsky:2008} showed, using a counting argument based on \cite{deloera:2005}, that the restriction of the underlying action on $F$ reveals extra information on the rate of growth of $M^\prime_n$ if $p:=rank(F)< d$. More specifically, as $n \to \infty$,
\[
n^{-p/\alpha} M^\prime_n \Rightarrow \left\{ \begin{array}{ll}
                0   & \mbox{ if $\{\phi_t\}_{t \in F}$ is conservative},\\
                \mbox{Fr\'echet distribution} & \mbox{ if $\{\phi_t\}_{t \in F}$ is not conservative.}
             \end{array}
     \right.
\]
In the above set up, $p$ can be regarded as the effective dimension of the field. See also \cite{roy:2010a}, which investigates a deeper connection between effective dimension and extremes of discrete parameter stable fields.

This concept of effective dimension has not been extended to the continuous parameter case so far because in that case the group
$\{\phi_v:\,v \in \mathbb{R}^d\}$
of nonsingular transformations is not finitely generated and therefore the structure theorem of finitely generated Abelian groups cannot be used anymore. In this work, we remove this obstacle by observing that the effective dimensions of the discrete parameter subfields $\{X_t\}_{t \in 2^{-i}\bbz^d},\,i=0,1,\ldots$ are all equal (see Proposition~\ref{p2} below), which can be defined as the group theoretic dimension of $\mathbf{X}$. Our main goal is to investigate the relationship between this group theoretic dimension and the rate of growth of $M_t$.

\section{The Main Result}\label{sec:main:result}

In this section, we state the main theorem of this paper. As mentioned above, we would like to generalize the results given in the Section~5 of \cite{roy:samorodnitsky:2008} to the continuous parameter case. More specifically, we would like to extend Theorem~5.4 therein. In order to do that, we follow the notations in \cite{roy:samorodnitsky:2008} and introduce the notion of group theoretic dimension for the continuous parameter stationary S$\alpha$S random field $\mathbf{X}$.

For all $i \in \{0,1,2,\ldots\}$, we define
\begin{eqnarray*}
A_i&:=&\{\phi_u: u\in\Gamma_i\},
\end{eqnarray*}
where $\Gamma_i:=2^{-i}\bbz^d$ as defined in Section~\ref{sec:prelim}. By first isomorphism theorem of groups (see \cite{lang:2002}), it follows that
$
A_i \simeq \Gamma_i / K_i,
$
where
\begin{eqnarray*}
K_i&:=&\{u\in\Gamma_i:\phi_u=\one_S\}\,
\end{eqnarray*}
($\one_S$ denotes the identity map on $S$) is the kernel of the group homomorphism
$$
\Phi_i: \Gamma_i \to A_i
$$
defined by $\Phi_i(u)=\phi_u$, $u \in \Gamma_i$. In particular, as in \cite{roy:samorodnitsky:2008}, $A_i$ is a finitely generated Abelian group and hence by appealing to the structure theorem of such groups (see Theorem 8.5 in Chapter 1 of \cite{lang:2002}), we get $\bar F_i,\bar N_i\subset A_i$ such that $\bar F_i$ is a free Abelian group of rank $p_i \leq d$, $\bar N_i$ is a finite group,
$$
\bar F_i\cap\bar N_i=\{0\}\,,
$$
and
$$
A_i=\bar F_i+\bar N_i\,.
$$
Since $\bar F_i$ is free, there exists an injective group homomorphism $\Psi_i:\bar F_i\longrightarrow\Gamma_i$ so that the following diagram commutes:
\begin{equation}\label{cd}
\begin{CD}
\Gamma_i @>{\Phi_i}>>  A_i\\
@A{inclusion}AA @AA{inclusion}A\\
F_i @<<\Psi_i< \bar F_i
\end{CD}
\end{equation}
where $F_i:=\Psi_i(\bar F_i)$ is a free subgroup of $\Gamma_i$ of rank $p_i \leq d$.

Following \cite{roy:samorodnitsky:2008}, $p_i$ can be regarded as the effective dimension of the subfield $\{X_t\}_{t \in \Gamma_i}$. Recall that $p_i$ determines the rate of growth of the subfield. The first result of this paper shows that the $p_i$'s are all equal.

\begin{prop}\label{p2}
For all $i\ge0$, $p_i=p_{i+1}$.
\end{prop}

\begin{proof}Fix $i\ge0$. Since $K_i$ is a subgroup of $\mathbb{Z}^d$, it is a free Abelian group. Let $q_i:=rank(K_i)$. We start by showing that
\begin{equation}\label{p2.eq1}
p_i+q_i=d\,.
\end{equation}
Note that $F_i$ and $K_i$ are subgroups of $\Gamma_i$ of rank $p_i$ and $q_i$ respectively, with $F_i\cap K_i=\{0\}$. Thus, it is immediate that $p_i+q_i\le d$. To show the other inequality, it suffices to prove that
\begin{equation}\label{p2.eq2}
r\Gamma_i\subset F_i+K_i\,,
\end{equation}
where $r:=|\bar N_i|$. To that end, fix $u\in\Gamma_i$ and notice that $\Phi_i(u)$ being an element of $\Gamma_i/K_i$, can be written as
$$
\Phi_i(u)=\bar v+y
$$
where $\bar v\in\bar F_i$ and $y\in\bar N_i$. Since $|\bar N_i|=r$, it follows that $ry=0$, and hence
$$
\Phi_i(ru)=r\bar v\,.
$$
Define $v:=\Psi_i(r\bar v)$. Since the diagram in \eqref{cd} commutes, it follows that
$$
\Phi_i(v)=r\bar v=\Phi_i(ru)\,.
$$
Thus, $ru-v\in$ Ker$(\Phi_i)=K_i$, which shows \eqref{p2.eq2} and consequently proves \eqref{p2.eq1}.
In view of \eqref{p2.eq1}, it suffices to show that $q_i=q_{i+1}$ and that follows trivially because
$2K_{i+1}\subset K_i\subset K_{i+1}$.
\end{proof}

Based on the preceding result, we denote
$$
p:=p_0=p_1=\ldots \, \leq d
$$
and define it to be the group theoretic dimension of the random field $\mathbf{X}$. We shall assume throughout the paper that {\bf$p\ge1$} (see Remark~5.5 of \cite{roy:samorodnitsky:2008}). As in the discrete parameter case, $p$ has information on the rate of growth of the partial maxima \eqref{def:M_t} as described below in Theorem~\ref{extreme.main}, which extends Theorem~5.4 of \cite{roy:samorodnitsky:2008} to the continuous parameter case and is the main result of this paper.

\begin{theorem}\label{extreme.main}
(i) If the group action $\{\phi_t:t\in F_0\}$ is not conservative, then
\begin{equation}\label{main.eq1}
t^{-p/\alpha}M_t\Longrightarrow KZ_\alpha
\end{equation}
as $t\longrightarrow\infty$, where $K\in(0,\infty)$ is a constant and $Z_\alpha$ is the standard Fr\'echet-type extreme value random variable with c.d.f.
$$
P(Z_\alpha\le z)=\exp\left(-z^{-\alpha}\right)
$$
for $z>0$.\\
(ii) If the group action $\{\phi_t:t\in F_0\}$ is conservative, then
\begin{equation}\label{main.eq2}
t^{-p/\alpha}M_t\prob0
\end{equation}
as $t\longrightarrow\infty$.
\end{theorem}

\begin{remark}{\rm The constant $K$ in \eqref{main.eq1} can be written as
$$
K:=C_\alpha^{1/\alpha}K_X\,,
$$
where $C_\alpha$ is the stable tail constant given by
\begin{equation*}
C_\alpha = {\left(\int_0^\infty x^{-\alpha} \sin{x}\,dx
\right)}^{-1}
 =\left\{
  \begin{array}{ll}
  \frac{1-\alpha}{\Gamma(2-\alpha) \cos{(\pi
\alpha/2)}}&\mbox{\textit{\small{if }}}\alpha \neq 1,\\
  \frac{2}{\pi}
&\mbox{\textit{\small{if }}}\alpha = 1,
  \end{array}
 \right.
\end{equation*}
and $K_X$ is the constant that equals the limit in \eqref{limit:p1:noncons} below.
}
\end{remark}

\begin{remark} {\rm As described in \cite{samorodnitsky:2004a}, the maxima process grows in a smaller rate when $\{X_t\}$ has stronger dependence due to the conservativity of the underlying action. Therefore, stronger conservativity of the action yields a smaller value of $p$ (because of a bigger $K_0$) and hence smaller rate of growth of $M_t$ and this is manifested in Theorem~\ref{extreme.main}.}
\end{remark}

We shall prove Theorem~\ref{extreme.main} in the next section. We first illustrate it with the following continuous parameter analogue of Example~6.2 in \cite{roy:samorodnitsky:2008}.

\begin{example} \label{example}
{\rm Suppose $$U:=\{\zeta \in \mathbb{C}: |\zeta|=1\}$$ is the unit circle. We take $d=3$, and define the $\mathbb{R}^3$-action
$\{\phi_{(x,y,z)}\}$ on $S=\mathbb{R}\times U$ as
\[
\phi_{(x,y,z)}(s,\zeta)=(s+x-y, \zeta e^{i 2 \pi z})\,.
\]
Clearly, this action preserves the measure $\mu$ on $S$ defined as the product
of the Lebesgue measure on $\mathbb{R}$ and the Haar probability measure on
$U$. In parallel to Example~6.2 in \cite{roy:samorodnitsky:2008}, we can take any $f \in L^\alpha(S,\mu)$ and define a stationary $S\alpha S$
random field
$\{X_{(x,y,z)}\}$ by
\[
X_{(x,y,z)}=\int_{\mathbb{R}\times U}
f\big(\phi_{(x,y,z)}(s,\zeta)\big)\, d\tilde{M}(s,\zeta)\,,
\]
where $\tilde{M}$ is a $S\alpha S$ random measure on $\mathbb{R}\times U$ with control measure $\mu$.}

{\rm For all $i=0,1,2,\ldots$,
$
K_i=\{(x,y,z)\in \Gamma_i: x=y, z \in \mathbb{Z}\}
$
and therefore following the calculations in Example~6.2 of \cite{roy:samorodnitsky:2008}, we get
$
A_i \simeq 2^{-i}\mathbb{Z} \times (\mathbb{Z}/{2^i}\mathbb{Z})
$
and
$
F_i=2^{-i}\mathbb{Z} \times \{0\} \times \{0\}.
$
In particular, $p=1$ and $\{\phi_t\}_{t\in F_0}$ is not conservative since $W:=(0,1) \times U$ is a wandering set of positive $\mu$-measure. Therefore, Theorem~\ref{extreme.main} yields that
${t^{-1/\alpha}}M_t$
converges to a Fr\'echet type extreme value random variable.}
\end{example}

\begin{remark} \label{remark:criticism}
{\rm Theorem~\ref{extreme.main} above is expected to give better results than Theorem~4.1 of \cite{roy:2010b} when the underlying action is conservative. For example, in Example~\ref{example}, Theorem~4.1 of \cite{roy:2010b} would just yield $M_t=o_p(t^{3/\alpha})$ as opposed to $M_t=O_p(t^{1/\alpha})$. However, this is not always the case as shown in the following example.}
\end{remark}

\begin{example} \label{example:nadkarni}
{\rm Consider the continuous parameter analogue of Example~6.3 (based on an action suggested by M.~G.~Nadkarni) in \cite{roy:samorodnitsky:2008}: $S=\bbr$ endowed with the Lebesgue measure, $d=2$, $f:=I_{[0,1]}$, and for all $(u,v)\in\bbr^2$,
$$
\phi_{u,v}(s):=s+u-v\sqrt2,\,s\in\bbr
$$
and $X_{(u,v)}$ is defined by \eqref{intro.eq0} and\eqref{intro.eq1}. It can be shown that $M_t=O_p(t^{1/\alpha})$  in this example (see Remark \ref{remark:last} below) although both Theorem~\ref{extreme.main} above and Theorem~4.1 of \cite{roy:2010b} would give $M_t=o_p(t^{2/\alpha})$.}
\end{example}

\section{Proof of Theorem~\ref{extreme.main}} \label{sec:proof}

The rate of growth of the maxima process is determined by that of the deterministic function $b(T)$ defined as
$$
b(T):=\left\{\int_S\sup_{t\in[-T\one,T\one]\cap\Gamma}|f_t(s)|^\alpha\mu(ds)\right\}^{1/\alpha}.
$$
See \cite{samorodnitsky:2004a, samorodnitsky:2004b}, \cite{roy:samorodnitsky:2008}. We start by studying the growth rate of $b(T)$ in both the conservative and the non-conservative cases using the discrete parameter approximation of the field and appealing to the results available in Section~5 of \cite{roy:samorodnitsky:2008}.

\begin{prop}\label{p1} (i) If the action $\{\phi_t: t\in F_0\}$ is conservative, then,
$$
\lim_{T\to\infty}T^{-p/\alpha}b(T)=0\,.
$$
(ii) On the other hand, if $\{\phi_t:t\in F_0\}$ is not conservative, then
\begin{equation}
\lim_{T\to\infty}T^{-p/\alpha}b(T) \label{limit:p1:noncons}
\end{equation}
exists, and is finite and positive.
\end{prop}

In order to prove Proposition~\ref{p1}, we need the following lemmas, whose proofs are presented in the Appendix.

\begin{lemma}\label{l1} If for any finitely generated Abelian group $G$, the action $\{\phi_u:u\in G\}$ is conservative, then so is the action $\{\phi_{ru}:u\in G\}$ for all integers $r\ge1$.
\end{lemma}

\begin{lemma}\label{p0} Suppose that for some integer $I\ge0$, the action $\{\phi_t:t\in F_I\}$ is conservative. Then, for all $i\ge0$, the action $\{\phi_t:t\in F_i\}$ is conservative.
\end{lemma}

\begin{lemma}\label{l2} There exist a positive integer $M$, $u_1,\ldots,u_p\in F_0$ and\\
$v_1,\ldots,v_q\in K_0$ so that for all $n\ge 1$,
\begin{align}
&[-n\one,n\one]\cap\Gamma  \subset \Biggl\{y+\sum_{i=1}^p\alpha_iu_i+\sum_{j=1}^q\beta_jv_j:y\in[-M\one,M\one]\cap\Gamma, \label{l2.eq1}\\
&\;\;\;\;\;\;\;\;\;\;\;\;\;\;\;\;\;\;\;\;\;\;\;\;\;\;\;\;\;\;\;\;\;\;\;\;\;\;\;\;\;\;\alpha_i,\beta_j\in[-Mn,Mn]\cap\qbin\mbox{ for all }i,j\Biggr\}\,, \nonumber
\end{align}
where $q:=d-p$ and $\qbin:=\bigcup_{m=0}^\infty2^{-m}\bbz$ denotes the set of binary rationals.
\end{lemma}

\begin{proof} [Proof of Proposition \ref{p1}]
(i) Choose $M$, $u_1,\ldots,u_p$ and $v_1,\ldots,v_q$ so that \eqref{l2.eq1} holds. Define
\begin{align*}
E&:=\Biggl\{y+\sum_{i=1}^p\alpha_iu_i+\sum_{j=1}^q\beta_jv_j:y\in[-M\one,M\one]\cap\Gamma,\\
&\,\,\,\,\,\,\,\,\,\,\,\,\,\;\;\;\;\;\;\alpha_i\in[0,M]\cap\qbin,\beta_j\in[0,1]\cap\qbin\mbox{ for all }i,j\Biggr\}\\
\intertext{and the sets}
E_m&:=\Biggl\{y+\sum_{i=1}^p\alpha_iu_i+\sum_{j=1}^q\beta_jv_j:y\in[-M\one,M\one]\cap\Gamma_m,\\
&\,\,\,\,\,\,\,\,\,\,\,\,\,\;\;\;\;\;\;\alpha_i\in[0,M]\cap2^{-m}\bbz,\beta_j\in[0,1]\cap2^{-m}\bbz\mbox{ for all }i,j\Biggr\},\,m\ge0\,.
\end{align*}
Clearly, $E$ is a bounded subset of $\Gamma$. By Proposition 10.2.1 and Theorem 10.2.3 of \cite{samorodnitsky:taqqu:1994} it follows that
$$
\int_S\sup_{t\in E}|f_t(x)|^\alpha\mu(dx)<\infty\,.
$$
Fix $\vep>0$. Since $E_m\uparrow E$, by the monotone convergence theorem there exists $m\ge0$ so that
\begin{equation*}
\int_S\max_{t\in E_m}|f_t(x)|^\alpha\mu(dx)\ge\int_S\sup_{t\in E}|f_t(x)|^\alpha\mu(dx)-\vep\,.
\end{equation*}
Fix such an $m$. For $k_1,\ldots,k_p\in\bbz$, define the sets
\begin{align}
\overline H(k_1,\ldots,k_p)&:=\Biggl\{y+\sum_{i=1}^p\alpha_iu_i+\sum_{j=1}^q\beta_jv_j:y\in[-M\one,M\one]\cap\Gamma,\nonumber\\
&\;\;\;\;\;\;\;\;\;\;\;\;\alpha_i\in[k_iM,(k_i+1)M]\cap\qbin,\beta_j\in\qbin\mbox{ for all }i,j\Biggr\}\label{eq:defH}\\
\intertext{and}
H(k_1,\ldots,k_p)&:=\Biggl\{y+\sum_{i=1}^p\alpha_iu_i+\sum_{j=1}^q\beta_jv_j:y\in[-M\one,M\one]\cap\Gamma_m,\nonumber\\
&\;\;\;\;\;\;\;\;\;\alpha_i\in[k_iM,(k_i+1)M]\cap2^{-m}\bbz,\beta_j\in2^{-m}\bbz\mbox{ for all }i,j\Biggr\}\,.\nonumber
\end{align}

Define for $n\ge1$,
\begin{align*}
\overline a_n&:=\int_S\sup\left\{|f_t(x)|^\alpha:t\in\bigcup_{k_1=-n}^{n-1}\ldots\bigcup_{k_p=-n}^{n-1}\overline H(k_1,\ldots,k_p)\right\}\mu(dx)\\
\intertext{and}
a_n&:=\int_S\max\left\{|f_t(x)|^\alpha:t\in\bigcup_{k_1=-n}^{n-1}\ldots\bigcup_{k_p=-n}^{n-1}H(k_1,\ldots,k_p)\right\}\mu(dx).
\end{align*}
Clearly by \eqref{l2.eq1},
$
b(n)^\alpha\le\overline a_n
$.
Note that for $n\ge 1$,
\begin{align*}
&\overline a_n-a_n\\
&\le\sum_{k_1=-n}^{n-1}\ldots\sum_{k_p=-n}^{n-1}\int_S\Biggl[\sup_{t\in\overline H(k_1,\ldots,k_p)}|f_t(x)|^\alpha\nonumber\\
&\,\,\,\,\,\,\,\,\,\,\,\,\,\,\,\,\,\,\,\,\,\,\,\,\,\,\,\,\,\,\,\,\,\,\,\,\,\,\,\,\,\,\,\,\,-\max_{t\in H(k_1,\ldots,k_p)}|f_t(x)|^\alpha\Biggr]\mu(dx).\\
\intertext{Note that all the $H(k_1,\ldots,k_p)$'s are simply translates of each other and therefore using Corollary~4.4.6 of \cite{samorodnitsky:taqqu:1994}, the right hand side of the above inequality equals}
&(2n)^p\int_S\Biggl[\sup_{t\in\overline H(0,\ldots,0)}|f_t(x)|^\alpha-\max_{t\in H(0,\ldots,0)}|f_t(x)|^\alpha\Biggr]\mu(dx)\\
&=(2n)^p\int_S\Biggl[\sup_{t\in E}|f_t(x)|^\alpha-\max_{t\in E_m}|f_t(x)|^\alpha\Biggr]\mu(dx) \le(2n)^p\vep\,.
\end{align*}

The above computations put together imply that
\begin{equation}\label{p1.eq13}
\limsup_{n\to\infty}n^{-p}b(n)^\alpha\le2^p\vep+\limsup_{n\to\infty}n^{-p}a_n\,.
\end{equation}
The next step is to show that
\begin{equation}\label{p1.eq11}
\lim_{n\to\infty} n^{-p} a_n=0\,.
\end{equation}
To that end, notice that
\begin{eqnarray*}
a_n&=&\int_S\max\left\{|f_t(x)|^\alpha:t\in\bigcup_{k_1=-n}^{n-1}\ldots\bigcup_{k_p=-n}^{n-1}H^\prime(k_1,\ldots,k_p)\right\}\mu(dx)\,,
\end{eqnarray*}
where
\begin{align*}
&H^\prime(k_1,\ldots,k_p):=\Biggl\{y+\sum_{i=1}^p\alpha_iu_i+\sum_{j=1}^q\beta_jv_j:y\in[-M\one,M\one]\cap\Gamma_m,\\
&\;\;\;\;\;\;\;\;\;\;\;\;\;\;\;\;\alpha_i\in[k_iM,(k_i+1)M]\cap2^{-m}\bbz,\beta_j\in[0,1]\cap2^{-m}\bbz\mbox{ for all }i,j\Biggr\}\,,
\end{align*}
for $k_1,\ldots,k_p\in\bbz$.
Clearly, there exists a positive integer $c$ so that
$$
\bigcup_{k_1=-n}^{n-1}\ldots\bigcup_{k_p=-n}^{n-1}H^\prime(k_1,\ldots,k_p)\subset[-cn\one,cn\one]\cap\Gamma_m\,,
$$
for all $n\ge1$. Thus,
\begin{equation}\label{p1.eq12}
a_n\le\int_S\max_{t\in[-cn\one,cn\one]\cap\Gamma_m}|f_t(x)|^\alpha\mu(dx)\,.
\end{equation}
By Lemma \ref{p0}, the group action $\{\phi_t:t\in F_m\}$ is conservative. An appeal to Proposition 5.1 of \cite{roy:samorodnitsky:2008} shows that the right hand side of \eqref{p1.eq12} is $o(n^p)$, and thus proves \eqref{p1.eq11}. This along with \eqref{p1.eq13} and the fact that $\vep$ there is arbitrary shows that
$$
\lim_{n\to\infty}n^{-p/\alpha}b(n)=0\,.
$$
The proof follows from here by the observation that
\begin{equation}\label{p1.eq14}
T^{-p/\alpha}b(T)\le\left(\frac{\lceil T\rceil}T\right)^{p/\alpha}\lceil T\rceil^{-p/\alpha}b(\lceil T\rceil)
\end{equation}
for all $T>0$.\\

\noindent (ii) As in the proof of Part (i), fix $M$, $u_1,\ldots,u_p$ and $v_1,\ldots,v_q$ so that \eqref{l2.eq1} holds. For $k_1,\ldots,k_p,l_1,\ldots,l_q\in\bbz$, define the set
\begin{align*}
&\tilde H(k_1,\ldots,k_p,l_1,\ldots,l_q):=\Biggl\{y+\sum_{i=1}^p\alpha_iu_i+\sum_{j=1}^q\beta_jv_j:
y\in[-M\one,M\one]\cap\Gamma,\\
&\;\;\;\;\;\;\alpha_i\in[k_iM,(k_i+1)M]\cap\qbin,
\beta_j\in[l_jM,(l_j+1)M]\cap\qbin\mbox{ for all }i,j\Biggl\}\,.
\end{align*}
A restatement of \eqref{l2.eq1} is that for $n\ge 1$,
\begin{equation}\label{p3.eq1}
[-n\one,n\one]\cap\Gamma\subset\bigcup_{-n\le k_1,\ldots,k_p,l_1,\ldots,l_q\le n-1}\tilde H(k_1,\ldots,k_p,l_1,\ldots,l_q)\,.
\end{equation}

Let
\begin{align*}
&G_n:=\Bigl\{(k_1,\ldots,k_p)\in\bbz^p:-n\le k_i\le n-1\mbox{ for all }i\mbox{ and}\\\
&\,\,\,\,\,\,\,\,\,\,\,\,\;\;\;\;\;\;\;\tilde H(k_1,\ldots,k_p,l_1,\ldots,l_q)\cap[-n\one,n\one]\neq\phi\mbox{ for some }l_1,\ldots,l_q\in\bbz\Bigl\}.\\
\intertext{Define for $m,n\ge1$,}
&\tilde a_n^{(m)}:=\int_S\sup\Biggl\{|f_t(x)|^\alpha:t\in\bigcup_{(k_1,\ldots,k_p)\in G_n}\bigcup_{l_1,\ldots,l_q\in\bbz}\\
&\;\;\;\;\;\;\;\;\;\;\;\;\;\;\;\;\;\;\;\;\;\;\;\;\;\;\;\;\;\;\;\;\;\;\;\;\;\;\;\;\;\;\;\;\;\;\;\;\;\;\;\;\;\;\;\Gamma_m\cap \tilde H(k_1,\ldots,k_p,l_1,\ldots,l_q)\Biggr\}\mu(dx)\,,\\
\intertext{and}
&\tilde a_n:=\int_S\sup\Biggl\{|f_t(x)|^\alpha:t\in\bigcup_{(k_1,\ldots,k_p)\in G_n}\bigcup_{l_1,\ldots,l_q\in\bbz}\\
&\;\;\;\;\;\;\;\;\;\;\;\;\;\;\;\;\;\;\;\;\;\;\;\;\;\;\;\;\;\;\;\;\;\;\;\;\;\;\;\;\;\;\;\;\;\;\;\;\;\;\;\;\;\;\;\;\;\;\;\;\;\;\;\tilde H(k_1,\ldots,k_p,l_1,\ldots,l_q)\Biggr\}\mu(dx)\,.
\end{align*}

Let $K$ be an integer which is no smaller than the diameter (with respect to the $L^\infty$ norm) of $\tilde H(0,\ldots,0)$.
Let for $m\ge1$ and $T>0$,
$$
\tilde b_m(T):=\int_S\sup_{t\in[-T\one,T\one]\cap\Gamma_m}|f_t(x)|^\alpha\mu(dx)\,.
$$
We claim that the proof will follow if the following can be shown:
\begin{equation}\label{p3.eq2}
b(n)^\alpha\le \tilde a_n\le b(n+K)^\alpha\mbox{ for all }n\ge1\,,
\end{equation}
\begin{equation}\label{p3.eq3}
\tilde b_m(n)^\alpha\le \tilde a_n^{(m)}\le \tilde b_m(n+K)^\alpha\mbox{ for all }m,n\ge1\,,
\end{equation}
and
\begin{equation}\label{p3.eq4}
\lim_{m\to\infty}\limsup_{n\to\infty}n^{-p}\left(\tilde a_n-\tilde a_n^{(m)}\right)=0\,.
\end{equation}
Suppose, for the moment, that the above is true. An appeal to Lemma \ref{p0} shows that for all $m\ge0$, the group action $\{\phi_t:t\in F_m\}$ is non-conservative. By Proposition 5.1 in \cite{roy:samorodnitsky:2008} it follows that
$$
\lim_{n\to\infty}n^{-p}\tilde b_m(n)^\alpha=A_m\in(0,\infty)\mbox{ for all }m\ge1\,.
$$
This, along with \eqref{p3.eq3} yields that
\begin{equation}\label{p3.eq5}
\lim_{n\to\infty}n^{-p}\tilde a_n^{(m)}=A_m\mbox{ for all }m\ge1\,.
\end{equation}
We next show that
$
A:=\sup_{m\ge1}A_m<\infty
$.
To that end, observe that by \eqref{p3.eq4}, there exists $m\ge1$ so that
$$
\limsup_{n\to\infty}n^{-p}\left(\tilde a_n-\tilde a_n^{(m)}\right)\le1\,.
$$
Fix such an $m$. Thus
$$
n^{-p}\left(\tilde a_n-\tilde a_n^{(m)}\right)\le2
$$
for $n$ large enough, and hence
$$
\limsup_{n\to\infty}n^{-p}\tilde a_n\le2+A_m\,.
$$
Observe that
$$
A\le\limsup_{n\to\infty}n^{-p}b(n)^\alpha\le\limsup_{n\to\infty}n^{-p}\tilde a_n\,,
$$
the second inequality following from \eqref{p3.eq2}. Thus, $A<\infty$. Since for all $m,n$,
$
\tilde b_m(n)\le\tilde b_{m+1}(n)
$,
it follows that $A_m\le A_{m+1}$. Hence,
$$
\lim_{m\to\infty}A_m=A\,.
$$
In view of \eqref{p3.eq4} and \eqref{p3.eq5}, this implies that
$$
\lim_{n\to\infty}n^{-p}\tilde a_n=A\,.
$$
As a restatement of \eqref{p3.eq2}, we have that
$$
a_{n-K}\le b(n)^\alpha\le \tilde a_n
$$
for $n>K$. Hence,
$$
\lim_{n\to\infty}n^{-p}b(n)^\alpha=A\,.
$$
An observation similar to \eqref{p1.eq14} will prove Part (ii) of Proposition~\ref{p1}.

We now show \eqref{p3.eq2} -\;\eqref{p3.eq4}. The first inequality in \eqref{p3.eq2} follows trivially from the observation that
\begin{equation*}
[-n\one,n\one]\cap\Gamma\subset\bigcup_{(k_1,\ldots,k_p)\in G_n}\bigcup_{l_1,\ldots,l_q\in\bbz}\tilde H(k_1,\ldots,k_p,l_1,\ldots,l_q)\,,
\end{equation*}
which in turn is a consequence of \eqref{p3.eq1}. For the second inequality, it suffices to show that for all
$$
t\in\bigcup_{(k_1,\ldots,k_p)\in G_n}\bigcup_{l_1,\ldots,l_q\in\bbz}\tilde H(k_1,\ldots,k_p,l_1,\ldots,l_q)\,,
$$
there exists $t^\prime\in\Gamma$ with $\|t^\prime\|\le n+K$ so that
\begin{equation}\label{p3.eq7}
f_t(x)=f_{t^\prime}(x)\mbox{ for all }x\in S\,,
\end{equation}
where $\|\cdot\|$ denotes the $L^\infty$ norm on $\bbr^d$. By choice of $t$, there exist\\
$(k_1,\ldots,k_p)\in G_n$ and $l_1,\ldots,l_q\in\bbz$ so that
$$
t\in \tilde H(k_1,\ldots,k_p,l_1,\ldots,l_q)\,.
$$
Clearly,
$$
t=y+\sum_{i=1}^p\alpha_iu_i+\sum_{j=1}^q\beta_jv_j
$$
for some $y\in[-M\one,M\one]\cap\Gamma$,
$\alpha_i\in[k_iM,(k_i+1)M]\cap\qbin$ and $\beta_j\in[l_jM,(l_j+1)M]\cap\qbin$.
Since $(k_1,\ldots,k_p)\in G_n$, there exist $l_1^\prime,\ldots,l_q^\prime\in\bbz$ so that
\begin{equation}\label{p3.eq6}
[-n\one,n\one]\cap \tilde H(k_1,\ldots,k_p,l_1^\prime,\ldots,l_q^\prime)\neq\phi\,.
\end{equation}
Define
$$
t^\prime:=y+\sum_{i=1}^p\alpha_iu_i+\sum_{j=1}^q\left\{\beta_j+(l^\prime_j-l_j)M\right\}v_j\,.
$$
Clearly, \eqref{p3.eq7} holds for this $t^\prime$. Thus, for the second inequality in \eqref{p3.eq2}, it suffices to show that
$$
\|t^\prime\|\le n+K\,.
$$
Since \eqref{p3.eq6} holds, there exists
$$
s\in[-n\one,n\one]\cap \tilde H(k_1,\ldots,k_p,l_1^\prime,\ldots,l_q^\prime)\,.
$$
It is easy to see that the diameters of $\tilde H(k_1,\ldots,k_p,l_1^\prime,\ldots,l_q^\prime)$ and $\tilde H(0,\ldots,0)$ are the same because one is a translate of the other. Thus, the diameter of the former is bounded by $K$. Notice that from the definition of $t^\prime$, it is immediate that
$
t^\prime\in \tilde H(k_1,\ldots,k_p,l_1^\prime,\ldots,l_q^\prime)
$.
Since $s$ also belongs to that set, it follows that
$
\|s-t^\prime\|\le K
$.
Clearly $\|s\|\le n$ because $s\in[-n\one,n\one]$. This shows that $\|t^\prime\|\le n+K$, and thus proves \eqref{p3.eq2}.
The justification for \eqref{p3.eq3} follows by similar arguments.

Finally, we proceed to prove \eqref{p3.eq4}. Fix $\vep>0$. Fix $m_0\ge0$ such that
$$
\int_S\max_{t\in \tilde H(0,\ldots,0)\cap\Gamma_{m_0}}|f_t(x)|^\alpha\mu(dx)\ge\int_S\sup_{t\in \tilde H(0,\ldots,0)}|f_t(x)|^\alpha\mu(dx)-\vep\,.
$$
Note that for $m\ge m_0$ and $n\ge1$,
\begin{align*}
&\tilde a_n-\tilde a_n^{(m)}\\
&=\int_S\Biggl[\max_{(k_1,\ldots,k_p)\in G_n}\sup\Bigl\{|f_t(x)|^\alpha:t\in\bigcup_{l_1,\ldots,l_q\in\bbz}\tilde H(k_1,\ldots,k_p,l_1,\ldots,l_q)\Bigr\}-\\
&\max_{(k_1,\ldots,k_p)\in G_n}\sup\Bigl\{|f_t(x)|^\alpha:t\in\bigcup_{l_1,\ldots,l_q\in\bbz}\Gamma_m\cap \tilde H(k_1,\ldots,k_p,l_1,\ldots,l_q)\Bigr\}\Biggr]\mu(dx)\\
&=\int_S\Biggl[\max_{(k_1,\ldots,k_p)\in G_n}\sup\Bigl\{|f_t(x)|^\alpha:t\in \tilde H(k_1,\ldots,k_p,0,\ldots,0)\Bigr\}-\\
&\;\;\;\;\;\;\;\;\;\;\;\;\;\max_{(k_1,\ldots,k_p)\in G_n}\max\Bigl\{|f_t(x)|^\alpha:t\in\Gamma_m\cap \tilde H(k_1,\ldots,k_p,0,\ldots,0)\Bigr\}\Biggr]\mu(dx)\\
&\le\sum_{(k_1,\ldots,k_p)\in G_n}\int_S\Biggl[\sup_{t\in \tilde H(k_1,\ldots,k_p,0,\ldots,0)}|f_t(x)|^\alpha-\\
&\,\,\,\,\,\,\,\,\,\,\,\,\;\;\;\;\;\;\;\;\;\;\;\;\;\;\;\;\;\;\;\;\;\;\;\;\;\;\;\;\;\;\;\;\;\;\;\;\;\;\;\;\;\;\max_{t\in\Gamma_m\cap \tilde H(k_1,\ldots,k_p,0,\ldots,0)}|f_t(x)|^\alpha\Biggr]\mu(dx)\\
&\le(2n)^p\int_S\Biggl[\sup_{t\in \tilde H(0,\ldots,0)}|f_t(x)|^\alpha-\max_{t\in\Gamma_m\cap \tilde H(0,\ldots,0)}|f_t(x)|^\alpha\Biggr]\mu(dx) \le(2n)^p\vep\,.
\end{align*}
Thus,
$$
\limsup_{m\to\infty}\limsup_{n\to\infty}\,n^{-p}\left(\tilde a_n-\tilde a_n^{(m)}\right)\le2^p\vep\,.
$$
Since, $\vep$ is arbitrary, this shows \eqref{p3.eq4} and completes the proof.
\end{proof}

Having established Proposition~\ref{p1}, we are now ready to prove Theorem~\ref{extreme.main}. Let $K_X$ denote the limit obtained in \eqref{limit:p1:noncons}. We start by proving \eqref{main.eq1}. In view of Proposition \ref{p1}, it suffices to show that
\begin{equation}\label{main.eq3}
b_t^{-1}M_t\Longrightarrow C_\alpha^{1/\alpha}Z_\alpha
\end{equation}
as $t\longrightarrow\infty$. For $t>0$, let $\eta_t$ be a probability measure on $(S,{\mathcal S})$ with
$$
\frac{d\eta_t}{d\mu}(x)=b_t^{-\alpha}\sup_{s\in[-t\one,t\one]\cap\Gamma}|f_s(x)|^\alpha
$$
for all $x\in S$. Let $U_j^{(t)}$, $j=1,2$, be independent $S$-valued random variables with common law $\eta_t$. By Theorem 4.1 in \cite{roy:2010b}, in order to establish \eqref{main.eq3}, it suffices to show that
\begin{align}
&\lim_{t\to\infty}P\Bigg(\mbox{for some }s\in[-t\one,t\one]\cap\Gamma, \nonumber\\
&\;\;\;\;\;\;\;\;\;\;\;\;\;\;\;\;\;\;\;\;\;\;\;\;\;\;\frac{|f_s(U_j^{(t)})|}{\sup_{u\in[-t\one,t\one]\cap\Gamma}|f_u(U_j^{(t)})|}>\vep,\,j=1,2\Bigg)=0 \label{main.eq4}
\end{align}
for all $\vep>0$. Using arguments given in \cite{samorodnitsky:2004b} (see the proof of (2.14) implies (2.12) therein) and with $\overline H(k_1,\ldots,k_p)$ as in \eqref{eq:defH}, we have that the  probability in \eqref{main.eq4} is bounded by
\begin{eqnarray*}
&&\sum_{k_1=-\lceil t\rceil}^{\lceil t\rceil-1}\ldots\sum_{k_p=-\lceil t\rceil}^{\lceil t\rceil-1}P\Biggl(\mbox{for some }s\in\overline H(k_1,\ldots,k_p),\\
&&\,\,\,\,\,\,\,\,\;\;\;\;\;\;\;\;\;\;\;\;\;\;\;\;\;\;\;\;\;\;\;\;\;\;\;\;\;\;\;\;\frac{|f_s(U_j^{(t)})|}{\sup_{u\in[-t\one,t\one]\cap\Gamma}|f_u(U_j^{(t)})|}>\vep,\,j=1,2\Biggr)\\
&&\le \left(2\lceil t \rceil\right)^p\vep^{-2\alpha}b_t^{-2\alpha}\left(\int_S\sup_{s\in\overline H(0,\ldots,0)}|f_s(x)|^\alpha m(dx)\right)^2 \to 0
\end{eqnarray*}
as $t \to \infty$ by \eqref{limit:p1:noncons}. This shows \eqref{main.eq4} and hence completes the proof of \eqref{main.eq1}.

That \eqref{main.eq2} is true, follows easily from Theorem 4.1 of \cite{roy:2010b} and Proposition \ref{p1} by an argument similar to the proof of (2.7) in \cite{samorodnitsky:2004b}.

\begin{remark} \label{remark:last}
{\rm It is easy to see that the above proof can be carried over to any finitely generated subgroup $\Gamma_0$ of $\bbr^d$ of rank $d$ and $\Gamma_n:=2^{-n}\Gamma_0$ for $n\ge1$. The subgroup $\Gamma_0$ may be suitably chosen to yield better results in some cases. Consider Example~\ref{example:nadkarni} once again. Taking $\Gamma_0=\{(i\sqrt2,j):i,j\in\bbz\}$ and using the estimate obtained in Example 6.3 in \cite{roy:samorodnitsky:2008}, the sharper bound $M_t=O_p(t^{1/\alpha})$ follows although choosing $\Gamma_0=\bbz^2$ would just give $M_t=o_p(t^{2/\alpha})$.
\\}
\end{remark}

\noindent \textbf{Acknowledgment. }The authors are thankful to Gennady Samorodnitsky for some helpful discussions.

\appendix
\section{Proofs of the Lemmas}

\subsection{Proof of Lemma~\ref{l1}} If possible, let $A$ be a wandering set of positive measure for $\{\phi_{ru}:u\in G\}$. By Proposition 2.2 in \cite{roy:samorodnitsky:2008}, it follows that
\begin{equation}\label{l1.eq1}
\sum_{t\in G}\one_A\circ\phi_t=\infty\mbox{ a.e. on }A\,.
\end{equation}
Suppose that $G$ has a generating set of size $k$. Since $A$ is a wandering set for $\{\phi_{ru}:u\in G\}$, by arguments similar to those in the proof of Theorem 3.4, page 18 in \cite{krengel:1985}, it follows that
\begin{eqnarray}
\sum_{t\in G}\one_A\circ\phi_t&\le&r^k\,.\label{p0.eq2}
\end{eqnarray}
Clearly, \eqref{l1.eq1} and \eqref{p0.eq2} contradict each other. This completes the proof.

\subsection{Proof of Lemma~\ref{p0}} We first show that for $i\ge I+1$, the action $\{\phi_t:t\in F_i\}$ is conservative. Without loss of generality we can and do assume that $I=0$, because otherwise, the elements of $\Gamma_I$ can be relabeled to become $\bbz^d$. We show that $\{\phi_t:t\in F_1\}$ is a conservative action. By similar arguments, the result will follow for all $i$. All that needs to be shown is that given $A\in\mathcal S$ with $\mu(A)>0$ there exist $s,t\in F_1$ with $s\neq t$ so that
\begin{equation}\label{p1.eq1}
\mu\left(\phi_s(A)\cap\phi_t(A)\right)>0\,.
\end{equation}
Fix such an $A$.
Let $r:=|\bar N_1| \geq 1$. By Lemma \ref{l1} it follows that $\{\phi_{rt}:t\in F_0\}$ is a conservative action. Thus, there exist $s^\prime,t^\prime\in F_0$ with $s^\prime\neq t^\prime$ so that
\begin{equation}\label{p1.eq2}
\mu\left(\phi_{rs^\prime}(A)\cap\phi_{rt^\prime}(A)\right)>0\,.
\end{equation}
Since $s^\prime\in F_0\subset\Gamma_0\subset\Gamma_1$, $\Phi_1(s^\prime)\in\Gamma_1/K_1$. Since
$$\Gamma_1/K_1=\overline F_1+\bar N_1\,,$$
there exist $\bar s\in\overline F_1$ and $\bar y\in\bar N_1$ so that $\Phi_1(s^\prime)=\bar s+\bar y$. Using the fact that $r\bar y=0$, it follows that
$$
\Phi_1(rs^\prime)=r\bar s\,.
$$
Define
$$s:=\Psi_1(r\bar s)\,.$$
Clearly, $s\in F_1$. Using \eqref{cd} with $i=1$ yields that
$$
\Phi_1(s)=r\bar s=\Phi_1(rs^\prime)\,.
$$
Thus, $s-rs^\prime\in K_1$ and hence
$$
\phi_s=\phi_{rs^\prime}\,.
$$
By similar arguments it follows that $\Phi_1(rt^\prime)\in\overline F_1$, and hence one can define
$$
t:=\Psi_1\circ\Phi_1(rt^\prime)\,.
$$
Once again, the same arguments as above will show that
$$
\phi_t=\phi_{rt^\prime}\,.
$$
An appeal to \eqref{p1.eq2} shows that \eqref{p1.eq1} holds with this choice of $s$ and $t$. To complete the proof, all that needs to be checked is that $s\neq t$. This is immediate because if $s$ and $t$ are equal then so are $\Phi_1(rs^\prime)$ and $\Phi_1(rt^\prime)$ and therefore,
$
rs^\prime-rt^\prime\in K_1 \cap F_0 \subseteq K_0 \cap F_0 =\{0\}
$
yielding $s^\prime=t^\prime$, which is a contradiction. This proves the fact that for $i\ge I+1$, the action $\{\phi_t:t\in F_i\}$ is conservative.

We now show that the action $\{\phi_t:t\in F_i\}$ is conservative for $i\le I$. Since $\{\phi_t:t\in F_I\}$ is a conservative action, by (5.1) of \cite{roy:samorodnitsky:2008}, it follows that
$$
\lim_{n\to\infty}n^{-p}\int_S\sup_{t\in[-n\one,n\one]\cap\Gamma_I}|f_t(x)|^\alpha\mu(dx)=0\,.
$$
An immediate consequence of this is that
$$
\lim_{n\to\infty}n^{-p}\int_S\sup_{t\in[-n\one,n\one]\cap\Gamma_i}|f_t(x)|^\alpha\mu(dx)=0\,,
$$
for all $0\le i<I$. From here, it follows that $\{\phi_t:t\in F_i\}$ is a conservative action, for otherwise, (5.2) of \cite{roy:samorodnitsky:2008} would be contradicted, which can easily be seen to hold for non-conservative actions as well (not only for dissipative actions) by decomposing the action into conservative and dissipative parts (see \cite{aaronson:1997}). This proves Lemma~\ref{p0}.

\subsection{Proof of Lemma~\ref{l2}}
Lemma 5.2 of \cite{roy:samorodnitsky:2008} shows the existence of a positive integer $M^\prime$, $x_1,\ldots,x_l \in \bbz^d$, $u_1,\ldots,u_p\in F_0$ and $v_1,\ldots,v_q\in K_0$ so that for all $n\ge M^\prime$,
\begin{align}
&[-n\one,n\one]\cap\bbz^d\subset \bigcup_{k=1}^l\Biggl\{x_k+\sum_{i=1}^p\alpha_iu_i+\sum_{j=1}^q\beta_jv_j:\label{l2.eq0}\\
&\;\;\;\;\;\;\;\;\;\;\;\;\;\;\;\;\;\;\;\;\;\;\;\;\;\;\;\;\;\;\;\;\;\;\;\;\;\;\;\;\;\alpha_i,\beta_j\in[-M^\prime n,M^\prime n]\cap\bbz\mbox{ for all }i,j\Biggr\}\,. \nonumber\\
\intertext{Define $M:=M^\prime \vee \|x_1\|\vee \|x_2\|\cdots\vee \|x_l\|$. In order to establish \eqref{l2.eq1}, it is enough to show that for all $m \geq 0$,}
&[-n\one,n\one]\cap\Gamma_m\subset\Biggl\{y+\sum_{i=1}^p\alpha_iu_i+\sum_{j=1}^q\beta_jv_j:y\in[-M\one,M\one]\cap\Gamma_m, \nonumber\\
&\;\;\;\;\;\;\;\;\;\;\;\;\;\;\;\;\;\;\;\;\;\;\;\;\;\;\;\;\;\;\;\;\;\;\;\;\;\;\;\;\;\alpha_i,\beta_j\in[-Mn,Mn]\cap2^{-m}\bbz\mbox{ for all }i,j\Biggr\},\nonumber
\end{align}
which follows by induction on $m \geq 0$ from $\eqref{l2.eq0}$ and the observation that
$$
[-n\one,n\one]\cap \Gamma_{m+1}\subset \frac{1}{2}\left\{\left([-n\one,n\one]\cap\Gamma_m\right)+\left([-n\one,n\one]\cap\Gamma_m\right) \right\}.
$$

\end{document}